\documentclass[12pt]{amsart}
\usepackage{amssymb,amsmath,amsthm,latexsym}
\usepackage{mathrsfs}
\usepackage{a4wide}
\usepackage[usenames,dvipsnames]{color}
\usepackage[utf8]{inputenc}
\usepackage[T1]{fontenc}
\usepackage{dsfont}
\usepackage{tikz}
\usepackage{float}
\usepackage[mathscr]{euscript}
\usepackage{mathtools}
\usepackage{hyperref}

\newtheorem{theorem}{Theorem}[section]
\newtheorem*{maint}{{\textbf{Theorem}}}

\newtheorem{lemma}[theorem]{Lemma}
\newtheorem{corollary}[theorem]{Corollary}
\newtheorem{proposition}[theorem]{Proposition}
\theoremstyle{remark}

\newtheorem{remark}[theorem]{Remark}
\newtheorem*{remark*}{Remark}
\theoremstyle{definition}
\newtheorem{definition}{Definition}
\newtheorem{problem}{Problem}

\numberwithin{equation}{section}
\makeatother

\newcommand{\vertiii}[1]{{\left\vert\kern-0.25ex\left\vert\kern-0.25ex\left\vert #1 
		\right\vert\kern-0.25ex\right\vert\kern-0.25ex\right\vert}}

\newcounter{smallromans}

	{\end{list}}

\newcounter{smallromansdash}

	{\end{list}}

\newcounter{bigromans} 
	{\end{list}}

\title[Inverse problems for symmetric doubly stochastic matrices]{Inverse problems for symmetric doubly stochastic matrices whose Sule\u{\i}manova spectra\\ are bounded below by 1/2}
\dedicatory{In memoriam: Miroslav Fiedler (1926--2015)} 
\author{Micha{\l} Gnacik}
\address{School of Mathematics and Physics, Lion Gate Building, Lion Terrace, University of Portsmouth, Portsmouth, PO1 3HF, United Kingdom}
\email{michal.gnacik@port.ac.uk}
\author{Tomasz Kania}
\address{Institute of Mathematics, Czech Academy of Sciences, \v{Z}itn\'{a} 25, 115~67 Prague 1, Czech Republic and Institute of Mathematics, Jagiellonian University, {\L}ojasiewicza 6, 30-348 Krak\'{o}w, Poland}
\date{\today}
\email{kania@math.cas.cz, tomasz.marcin.kania@gmail.com}
\date{\today}

\begin{document}
	\begin{abstract}
		A new sufficient condition for a list of real numbers to be the spectrum of a symmetric doubly stochastic matrix is presented; this is a contribution to the classical spectral inverse problem for symmetric doubly stochastic matrices that is still open in its full generality. It is proved that whenever $\lambda_2, \ldots, \lambda_n$ are non-positive real numbers with $1 + \lambda_2 + \ldots + \lambda_n \geqslant 1/2$, then there exists a symmetric, doubly stochastic matrix whose spectrum is precisely $(1, \lambda_2, \ldots, \lambda_n)$. We point out that this criterion is incomparable to the classical sufficient conditions due to Perfect--Mirsky, Soules, and their modern refinements due to Nader \emph{et al.} We also provide some examples and applications of our results.  
	\end{abstract}
	\maketitle
	
	\section{Introduction}
	A square matrix with real entries is termed \emph{stochastic}, when it has all entries non-negative and each row adds up to 1. Stochastic matrices are conveniently interpreted as transition matrices of finite-state Markov chains (hence the terminology). The aim of this note is to consider the inverse eigenvalue problem for doubly stochastic matrices (also called \emph{bistochastic} in the literature; a stochastic matrix is \emph{doubly stochastic} if its transpose is stochastic too). Doubly stochastic matrices may be interpreted as transition matrices of finite-state symmetric Markov chains. Permutation matrices are a paradigm example of a class of doubly stochastic matrices; according to the Birkhoff–-von Neumann theorem, the set of $n\times n$ doubly stochastic matrices is the convex hull of the set of permutation matrices of an $n$-element set.\smallskip
	
	Inverse eigenvalue problems for classes of matrices such as (symmetric or not) matrices with non-negative entries, stochastic, or doubly stochastic are well-rooted in the literature, having their origin in the works of Sule\u{\i}manova \cite{Sul1, Sul2} and, independently, Perfect (\cite{Pe52, Pe53}) with an important subsequent continuation by Perfect and Mirsky \cite{PerfM}. Recently, the problem has gained new impetus as reflected by a plethora of new sufficient conditions (\cite{ Ccs,HwP, JoPa, Lei, MAMGK, nader, XuLei}). We refer to Mourad's paper \cite{Mourad} for a good overview concerning the said problems.\smallskip
	
	More explicitly, the symmetric doubly stochastic eigenvalue inverse problem (SDIEP) asks the following: \smallskip
	\begin{center}\emph{Let $\lambda_1, \lambda_2, \ldots, \lambda_n$ be a list of real numbers. In what circumstances does there exist a~symmetric doubly stochastic matrix whose spectrum consists of these numbers? }\end{center}\smallskip
	
	Quite naturally, the inverse problems for other classes of matrices are formulated totally analogously. In this short note, we focus only on providing a new sufficient condition to address SDIEP. \smallskip
	
	Since any matrix $A$ solving the above problem has non-negative entries, by the classical Frobenius--Perron theorem, $A$ must have a
	non-negative eigenvalue $\lambda_1$ (that is called \emph{the Perron eigenvalue} of $A$) such that $\lambda_1 \geqslant |\lambda|$ for any other eigenvalue $\lambda$ of $A$ (and the eigenvector associated to $\lambda_1$ has non-negative entries, it is also known as \emph{the Perron eigenvector}). It is to be noted that already for stochastic matrices, $\lambda_1=1$ is the Perron eigenvalue to which corresponds the eigenvector comprising only $1$s. Consequently, without loss of generality we shall restrict ourselves to $\lambda$s from the interval $[-1,1]$.\smallskip
	
	Since the trace of (any power of) a square matrix with non-negative entries is non-negative, for a list $(1, \lambda_2, \ldots, \lambda_n)$ of real numbers to form a spectrum of a~solution to SDIEP, it is necessary that $1 + \lambda_2^k + \ldots + \lambda_n^k \geqslant 0$ for any $k\in \mathbb N$. Perfect and Mirsky \cite{PerfM} provided a~fairly general sufficient condition for the possibility of solving SDIEP for a list of real numbers $1=\lambda_1 \geqslant \lambda_2 \geqslant \ldots \geqslant -1$. Namely, they proved that as long as
	\begin{equation}\label{eqn: permir} \frac{1}{n} +\frac{1}{n(n-1)}\lambda_2 + \ldots + \frac{1}{2\cdot 1}\lambda_{n} \geqslant 0,\end{equation}
	there exists a symmetric doubly stochastic $n\times n$ matrix whose spectrum is equal to $\{1, \lambda_2, \ldots, \lambda_n\}$. This condition was subsequently refined by Soules \cite{Soules}, who gave a finer condition depending on the remainder of $n$ when divided by 2. More specifically, let $m$ be such that $n=2m +1$ in the case $n$ is odd and $n=2m+2$ in the case where $n$ is even. If
	\begin{equation}\label{eqn: soules}\frac{1}{n} +\frac{n-m-1}{n(m+1)}\lambda_2  + \sum_{k=1}^n\frac{1}{(k+1)k}\lambda_{n-2k+2} \geqslant 0,\end{equation}
	then there exists a symmetric doubly stochastic $n\times n$ matrix whose spectrum coincides with $\{1, \lambda_2, \ldots, \lambda_n\}$.\smallskip
	
	Soules' condition was refined further by Nader \emph{et al.}~ who arrived at~a more complicated condition that depends on the remainder $n \! \mod 4$ (\cite[Theorem 5]{nader}) that covers a large class of cases.\smallskip
	
	Following the terminology introduced by Paparella \cite{Pap}, we call a list of real numbers $\sigma = (\lambda_1 = 1, \lambda_2, \ldots, \lambda_n)$  a \emph{Sule\u{\i}manova spectrum}, whenever $\lambda_j\leqslant 0$ for $j=2, \ldots, n$ and $1 + \lambda_2 + \ldots + \lambda_n \geqslant 0$. Moreover, we say that Sule\u{\i}manova spectrum is \emph{normalised} if  $\lambda_2 \geqslant \cdots \geqslant \lambda_n.$
	
	\smallskip
	
	Fiedler (\cite[Theorem 2.4]{Fiedler}) showed that normalised Sule\u{\i}manova spectra are realisable by symmetric non-negative matrices. Soto and Ccapa (\cite[Theorem 3.3]{Soho}) proved that as long as $\lambda_j$ < 0 for $j=2, \ldots, n$ every normalised Sule\u{\i}manova spectrum is realisable as a~spectrum of a stochastic matrix (not necessarily symmetric). \smallskip
	
	Johnson and Paparella (\cite[Problem 6.2. and Theorem 6.3.]{JoPa}) proved, among other things, that every normalised Sule\u{\i}manova spectrum is realisable as a spectrum of a symmetric doubly stochastic matrix for all Hadamard orders (recall that the order of a Hadamard matrix must be $1$, $2$ or a multiple of $4$). In particular, when $n=2^k$ for some $k$, the resulting matrix $A$ is trisymmetric (\cite[Corollary 6.5]{JoPa}), that is, it satisfies any two of the following three conditions: \emph{symmetric}, \emph{persymmetric} ($AK = KA^T$), or \emph{centrosymmetric} ($AK = KA$), 
	where $K$ is the exchange matrix (backward identity); satisfying any two of the above three conditions always implies that the remaining third condition holds. \smallskip
	
	It is to be noted that a normalised Sule\u{\i}manova spectrum with $1 + \lambda_2 + \ldots + \lambda_n = 0$ may fail to be realisable within the class of symmetric doubly stochastic matrices. Indeed, if $n$ is an odd number, \emph{e.g.}, the list $(1, 0, 0, \ldots, 0, -1)$ cannot be a~spectrum of any symmetric doubly stochastic matrix  (\cite[Corollary 1]{nader}).\smallskip
	
	%We aim to investigate normalised Sule\u{\i}manova spectra under the extra condition $\lambda_j\leqslant 0$ for $j=2, \ldots, n$ to only negative $\lambda$s.\smallskip
	For the sake of brevity, we introduce the following piece of terminology.
	\begin{definition}\label{def:sul}
		A list $\sigma=(1,\lambda_2, \ldots, \lambda_n)$ is called a $\delta$-\emph{Sule\u{\i}manova spectrum} ($\delta > 0$), whenever
		\begin{equation}\label{eqn: suleimanova}1+ \lambda_2 + \ldots + \lambda_n \geqslant \delta \end{equation}
		and $\lambda_j\leqslant 0$ for $j=2, \ldots, n$.\smallskip
	\end{definition}
	
	The main result of this note is to prove that $1/2$-Sule\u{\i}manova spectra may be indeed realised as spectra of symmetric, doubly stochastic matrices. Note that we drop the assumption for the spectrum to be normalised.\smallskip
	%In some sense, this result constitutes a common roof for the above-mentioned results of Fiedler (\cite[Theorem 2.4]{Fiedler} and Soto--Ccapa (\cite[Theorem 3.3]{Soho}).
	
	\begin{maint}Let $\sigma = ( 1, \lambda_2, \ldots, \lambda_n)$ be a $1/2$-Sule\u{\i}manova spectrum. Then there exists a~symmetric, doubly stochastic matrix whose spectrum is precisely $\sigma$.\end{maint}
	
	\iffalse
	\begin{remark} \label{rem: posmat}
		%Since positive-definite matrices and the corresponding reproducing kernel Hilbert spaces play an important role in machine learning, in some classical least square problems \cite[Chapter 8]{Pau},
		It may be worth mentioning that the above result has a natural counterpart when all the eigenvalues are non-negative. That is, whenever $\lambda_j \geqslant 0$ for $j =2, \ldots, n$
		one may want to consider the condition 
		\begin{equation}\label{eqn: sul2}  1+ \lambda_2 + \ldots + \lambda_n \leqslant \gamma,\end{equation}
		where $\gamma \geqslant 1$.
		In this case, we show that if $\sigma = (1, \lambda_2, \ldots, \lambda_n)$ satisfies (\ref{eqn: sul2}) with $\gamma = \frac{3}{2}$, then there is a symmetric, doubly stochastic matrix with $\sigma$ as its spectrum. 
	\end{remark}
	\fi
	
	We discern that the idea for the proof has its origins in the theory of Markov chains.
	Roughly speaking, we consider a simple (symmetric) random walk on the discrete torus $\mathbb Z / n\mathbb Z$ ($n \in \mathbb{N}$) represented by the following graph that we denote by $S_n$:
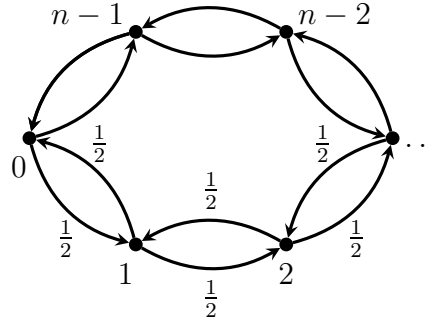
	\begin{center}
		\begin{tikzpicture}[scale=1.0, transform shape, >=stealth, node distance=2cm, very thick]
		\tikzstyle{nodeStyle} = [draw,fill,shape=circle, ,circle,inner sep=0pt,minimum size=4pt]
		\node[nodeStyle] (0)              {};
		\node[nodeStyle] (1) [below right of=0] {};
		\node[nodeStyle] (2) [right of=1] {};
		\node[nodeStyle] (3) [above right of=2] {};
		\node[nodeStyle] (4) [above left of=3] {};
		\node[nodeStyle] (5) [left of=4] {};
		\node[below] at (0.south) {$0$\ \ \ };
		\node[below] at (1.south) {$1$\ \ \ };
		\node[below] at (2.south) {$2$};
		\node[right] at (3.south) {$\ldots$};
		\node[above right] at (4.south) {$n-2$};
		\node[above left] at (5.south) {$n-1$};
		\path 
		(1) edge [->, bend right] node[below] {$\frac{1}{2}$} (2)
		(0) edge [->, bend right] node[below] {$\frac{1}{2}$} (1)
		(1) edge [->, bend right] node[above] {$\frac{1}{2}$} (0)
		(2) edge [->, bend right] node[above] {$\frac{1}{2}$} (1)
		(3) edge [->, bend right] node[above] {$\frac{1}{2}$} (2)
		(2) edge [->, bend right] node[below] {$\frac{1}{2}$} (3)
		(3) edge [->, bend right]  (4)
		(4) edge [->, bend right]  (3)
		(4) edge [->, bend right]  (5)
		(5) edge [->, bend right]  (4)
		(5) edge [->, bend right] (0)
		(0) edge [->, bend right] (5)
		(5) edge [->, bend right] (0);
		\end{tikzpicture}
	\end{center} 
	
	The crux of the proof is to use the eigenvectors of the transition probability matrix $P_n$ of the above Markov chain as a \emph{generator} for a class of new symmetric, doubly stochastic matrices. More specifically, we construct the desired matrices as Schur forms of the type $Q\Lambda Q^T$, where the column vectors of $Q$ are suitably chosen eigenvectors of $P_n$ and $\Lambda$ is a~diagonal matrix containing the Sule\u{\i}manova spectrum.  
	This very idea is not specific to $Q$, however it exploits the properties of $Q$, that is, symmetry and orthogonality. Such approach offers a room for improvement---we address this further in Remark \ref{rem: household}.  \smallskip
	
	By appealing to this heuristics and backed with some numerical evidence, we raise the following problem.
	
	\begin{problem}\label{problem} Let $\delta > 0$. Can every $\delta$-Sule\u{\i}manova spectrum be realised as the spectrum of a symmetric doubly stochastic matrix?
	\end{problem}
	The main result of the note asserts that the answer is positive for $\delta \geqslant 1/2$.
	
	\section{Proof of the main result}
	We consider the symmetric random walk on the graph $S_n$ described in the Introduction. The corresponding transition probability matrix $P_n$ is then given by 
	$$
	P_n  = \begin{bmatrix} 
	0 &  \frac{1}{2}& 0&0&\cdots &0&0& \frac{1}{2}\\
	\frac{1}{2} &0&   \frac{1}{2} & 0&\cdots &0&0& 0 \\
	&\ddots& \ddots & \ddots  & & & \\
	0 &  0& 0&0&\cdots &\frac{1}{2} &0& \frac{1}{2}\\
	\frac{1}{2} &  0& 0&0 &\cdots &0 &\frac{1}{2} &0\\
	\end{bmatrix}.
	$$
	Note that $P_n$ is a symmetric circulant matrix. \smallskip
	
	\noindent \emph{Notation.}
	In this section, we index the eigenvalues to start from $0$ rather than $1$ so that
	$\sigma = (\lambda_0, \lambda_1, \ldots, \lambda_{n-1})$ rather than 
	$\sigma = (\lambda_1, \lambda_2, \ldots, \lambda_{n})$ as done in the previous section. This assures that the formulae in this section are clearer and more compact. \bigskip
	
	A rote calculation shows that the eigenvalues $\lambda_k$ and the corresponding (real) eigenvectors $u_k$ ($k\in \{0, \ldots, n-1\}$) of $P_n$  are given by
	\begin{equation}\label{wartosciwlasne}
	\lambda_k = \cos\Big(\frac{2\pi k}{n} \Big),\qquad  u_k^{(j)}= \cos\Big( \frac{2\pi kj}{n} \Big) \qquad (k, j \in \{0, \ldots, n-1\});
	\end{equation}
	we prove it, among other things, in Lemma \ref{wk}. Due to symmetry, we have $\lambda_k = \lambda_{n-k}$ and $u_k = u_{n-k}$, so the eigenvectors $u_k$ ($k\in \{0, \ldots, n-1\}$) fail to span $\mathbb{R}^n$; in the next lemma we also find another set of eigenvectors corresponding to $P_n$ that do indeed form an orthonormal basis of $\mathbb{R}^n$. For an elaborate discussion on eigenvalues of transition matrices we refer the reader to  \cite[chapter 12]{Levin}.

	\begin{lemma}\label{wk} \enskip
		\begin{itemize}
			\item[i)] The eigenvalues $\lambda_k$ and the corresponding eigenvectors  $u_k$ ($k\in \{0, \ldots, n-1\}$) of $P_n$ are given by \emph{(\ref{wartosciwlasne})}.
			\item[ii)] Set
			$$	 w_k^{(j)}= \sqrt{\frac{2}{n}} \sin\left( \frac{2\pi kj}{n}+\frac{\pi}{4} \right) \qquad (k, j \in \{0, \ldots, n-1\}).$$
			Then $w_k = [w_k^{(j)}]_{0\leqslant j\leqslant n-1}$ are eigenvectors of $P_n$ corresponding to $\lambda_0, \ldots, \lambda_{n-1}$. 
			Furthermore, these eigenvectors form an orthonormal basis of $\mathbb{R}^n$. 
		\end{itemize}
	\end{lemma}
	
	\begin{proof} \enskip 
		\begin{itemize}
			\item [i)]
			Since $P_n$ is a circulant matrix, it is diagonalisable by a  Fourier matrix $F_n$, in particular,
			$P_n =
			F_n^* \mbox{diag}(\lambda_0, \ldots, \lambda_{n-1}) F_n,$
			where  $F_n^*= \frac{1}{\sqrt{n}}[\varphi_0 \ \varphi_1 \ \ldots \  \varphi_{n-1}]$,
			$$\varphi_k = (1,e^{\frac{2\pi ik}{n} },\ldots, e^{\frac{2\pi ikj}{n} } \ldots, e^{\frac{2\pi ik(n-1)}{n}} )^{\rm T}\qquad (k\in \{0, \ldots, n-1\})$$
			and $\lambda_k$ are given by (\ref{wartosciwlasne})
			(see, \emph{e.g.}, \cite[Theorem 3.2.2]{Dav}). Thus, each $\varphi_k$ is a (complex) eigenvector of $P_n$ corresponding to the eigenvalue $\lambda_k$. Since both the real and imaginary parts of $\varphi_k$ are eigenvectors of $P_n$ too, we conclude that $u_k^{(j)}= \cos\left( \frac{2\pi kj}{n} \right) $ and  $v_k^{(j)}= \sin\left( \frac{2\pi kj}{n} \right) $ are the coordinates of the (real) eigenvectors $u_k$ and $v_k$ corresponding to $\lambda_k$ for  $k,j \in \{0, \ldots, n-1\}$.
			\item[ii)]
			However, it is to be noted that neither of the systems $(u_k)_{k=0}^{n-1}$, $(v_k)_{k=0}^{n-1}$ spans $\mathbb{R}^n$. As sums of eigenvectors corresponding to the same eigenvalue, if non-zero, are still eigenvectors, let us consider the eigenvector $u_k + v_k$ corresponding to $\lambda_k$. Evidently, 
			$$ u_k^{(j)}+v_k^{(j)} = \sqrt{2} \sin\left( \frac{2\pi kj}{n}+\frac{\pi}{4} \right). $$
			We then define the vectors $w_k$ by
			$$w_k^{(j)}= \sqrt{\frac{2}{n}} \sin\left( \frac{2\pi kj}{n}+\frac{\pi}{4} \right)\qquad (k, j \in \{0, \ldots, n-1\}).$$ Conspicuously, $w_k$ is a (real) eigenvector of $P_n$ corresponding to $\lambda_k$. \smallskip %It is a straightforward exercise to check that $\|w_k\| = 1$ and that $(w_k)$ form the orthonormal basis of $\mathbb{R}^n$.
			
			In order to show that $(w_k)_{k=0}^{n-1}$ are pairwise orthogonal unit vectors, we invoke the identities $2\sin(x)\sin(y) = \cos(x-y) - \cos(x+y)$ and $\cos\left(x + \frac{\pi}{2}\right) = - \sin(x)$. Using these we arrive at
			\begin{align*}
			\left< w_k, w_l\right> =& \frac{2}{n} \sum_{j=0}^{n-1} \sin\left( \frac{2 \pi k j}{n} + \frac{\pi}{4} \right) \sin\left( \frac{2 \pi l j}{n} + \frac{\pi}{4} \right) \\
			=&  \frac{1}{n} \left( -1 + \sum_{j=0}^n \cos\left( \frac{2 \pi j(k-l)}{n} \right)\right) + \frac{1}{n}  \sum_{j=0}^n \sin\left( \frac{2 \pi j(k+l)}{n} \right),
			\end{align*}
			%\begin{align*}
			%\| w_k \|^2 =& \frac{2}{n} %\sum_{j=0}^{n-1} \sin^2\left( %\frac{2 \pi k j}{n} + \frac{\pi}{4} %\right)\\
			%=& 1 + \frac{1}{n} \sum_{j=0}^n \sin %\left( \frac{4 \pi k j}{n}\right).
			%\end{align*}
			where $\left< \cdot, \cdot\right>$ denotes the standard inner product in $\mathbb{R}^n$. 
			
			Let $m$ be a positive integer. Then
			$\sum_{j=0}^n \cos\left( \frac{2 \pi m j}{n}\right)$ and
			$\sum_{j=0}^n \sin \left( \frac{2 \pi m j}{n}\right)$ are the real and imaginary parts, respectively, of the expression
			$$ \sum_{j=0}^n e^{i \frac{2 \pi m j}{n}} = \frac{ e^{2i \pi \left(m+  \frac{m}{n}\right)}-1}{e^{2i \pi \frac{m}{n}}-1} = 1$$
			%\frac{ \overbrace{e^{2i \pi m}}^{=1}e^{2 i\pi \frac{m}{n}}-1}{e^{2i \pi \frac{m}{n}}-1} =1,$$
			respectively. Consequently, $\langle w_k, w_l\rangle = \delta_{k,l}$, where $\delta_{k, l}$ denotes the Kronecker delta, and thus $(w_k)_{k=0}^{n-1}$ forms an orthonormal basis of $\mathbb{R}^n$.
		\end{itemize}
	\end{proof}
	
	Before we proceed, a piece of notation is required. Let $Q$ be the matrix whose columns are precisely the eigenvectors from the statement of Lemma~\ref{wk}, that is, $Q = [w_0 \ w_1 \ \ldots \ w_{n-1} ]$.
	For a diagonal matrix $\Lambda = \mbox{diag}(\lambda_0 = 1, \lambda_1, \ldots, \lambda_{n-1})$ with $\lambda_k \in [-1,1]$ ($k \in \{1, \ldots, n-1\}$) we consider the product matrix
	\begin{equation}\label{eqn: RWABS} P(\Lambda):= Q \Lambda Q^{\rm T} = \sum_{j=0}^{n-1} \lambda_{j}w_j w_j^T.\end{equation}
	
	\begin{remark}
		We note that the matrix $P(\Lambda)$ is symmetric and its each row adds up to $1$ since $P(\Lambda) w_0 = w_0$. 
	\end{remark}
	Now we present the general form of the entries of $Q \Lambda Q^T$, with an arbitrary orthogonal $Q$ so that 
	$Q = [q_0 \ q_1 \ \ldots \ q_{n-1}]$
	with $q_0 = \frac{1}{\sqrt{n}} (1, \ldots, 1)^T$ and $q_k \in \mathbb{R}^n$ for each $k=1, \ldots, n-1$.
	\begin{remark}
		Let us denote the entry of $Q$ in the $k$\textsuperscript{th} row and $l$\textsuperscript{th} column by
		$q_{l}^{(k)}$ for $j, k \in \{0, \ldots, n-1\}$. 
		Note that $Q = [q_l^{(k)}]_{k,l=0}^{n-1}$ and  $Q\Lambda = [\lambda_{l} q_l^{(k)} ]_{k,l}$, where $\lambda_0 = 1$. 
		Hence,
		\begin{align} \label{eqn: mat_gen_entry}
		(Q\Lambda Q^T)_{k,l} =&  \sum_{j=0}^{n-1} \lambda_{j}q_j^{(k)}q_{j}^{(l)} 
		= \frac{1}{\sqrt{n}}q_0^{(k)}+ \sum_{j=1}^{n-1} \lambda_{j}q_j^{(k)}q_{j}^{(l)}.
		\end{align}
		Now let $Q$ be a matrix whose columns are precisely the eigenvectors of $P_n$ from the statement of Lemma~\ref{wk}, namely, $Q = [w_0 \ w_1 \ \ldots \ w_{n-1} ]$, then
		\begin{align}\label{eqn: mat_ent} q_j^{(k)} = w_j^{(k)}= \sqrt{\frac{2}{n}} \sin\left( \frac{2\pi kj}{n}+\frac{\pi}{4} \right)\qquad (k, j \in \{0, \ldots, n-1\}),
		\end{align}
		in particular, $q_0^{(k)} = w_0^{(k)} = \frac{1}{\sqrt{n}}$. 
		%In particular, 
		%\begin{align}\label{eqn: matrix_ent}
		%(Q\Lambda Q^T)_{k,l} =&  \frac{1}{n}\left( 1 + 2\sum_{j=1}^{n-1}\lambda_j\sin\left( \frac{2\pi kj}{n}+\frac{\pi}{4} \right) \sin\left( \frac{2\pi lj}{n}+\frac{\pi}{4} \right) \right).
		%\end{align}
		%So the result depends on the particular orthogonal matrix (orthonormal basis) that we select. 
	\end{remark}
	To summarise let us record the form of the entries of the matrix $P(\Lambda)$ in a lemma:
	\begin{lemma}
		\label{lemma: matrixform}
		We have $P(\Lambda)=[p_{kl}]_{k,l=0}^{n-1}$, where
		$$p_{kl} = \frac{1}{n}\left(1+ 2 \sum_{j=1}^{n-1}\lambda_j  \sin\left( \frac{2\pi kj}{n}+\frac{\pi}{4} \right) \sin\left( \frac{2\pi lj}{n}+\frac{\pi}{4} \right)\right)\qquad (k, l \in \{0,\ldots n-1\}).$$
	\end{lemma}
	\begin{proof}
		Follows from (\ref{eqn: mat_gen_entry}) and (\ref{eqn: mat_ent}).
		\iffalse
		Note that $Q = [w_l^{(k)}]_{k,l=0}^{n-1}$ and  $Q\Lambda = [\lambda_{l} w_l^{(k)} ]_{k,l}$, where $\lambda_0 = 1$. 
		Hence, 
		\begin{align*}
		(Q\Lambda Q^T)_{k,l} =&  \sum_{j=0}^{n-1} \lambda_{j}w_j^{(k)}w_{j}^{(l)} \\
		=& \frac{1}{n}\left( 1 + 2\sum_{j=1}^{n-1}\lambda_j\sin\left( \frac{2\pi kj}{n}+\frac{\pi}{4} \right) \sin\left( \frac{2\pi lj}{n}+\frac{\pi}{4} \right) \right).
		\end{align*}
		\fi
	\end{proof}
	\begin{proposition} \label{final}	Let $P(\Lambda)$ be as in \eqref{eqn: RWABS}.
		The matrix $P(\Lambda)$ is doubly stochastic if and only if
		\begin{equation}
		\label{eqn: inequality}
		\sum_{j=1}^{n-1}\lambda_j \sin\left( \frac{2\pi kj}{n}+\frac{\pi}{4} \right) \sin\left( \frac{2\pi lj}{n}+\frac{\pi}{4} \right) \geqslant -\frac{1}{2}
		\end{equation}
		for all $k \in \{0,\ldots n-1\}$ and $l \in \{k, \ldots, n-1\}$.
	\end{proposition}
	\begin{proof}
		The matrix $P(\Lambda)$ is doubly-stochastic if and only if 
		$p_{kl} \geqslant 0$
		for all $k \in \{0,\ldots n-1\}$ and $l \in \{k, \ldots, n-1\}$ and this is indeed equivalent to \eqref{eqn: inequality} by Lemma \ref{lemma: matrixform}.
	\end{proof}
	Finally, the main result  follows directly from the next corollary to Proposition~\ref{final}.

	\begin{corollary}\label{glowne}
		Let $P(\Lambda)$ be as in \eqref{eqn: RWABS}.
		\begin{enumerate}
			\item Suppose that $\lambda_i \leqslant 0$ for all $i \in \{1, \ldots, n-1\}$. Then $P(\Lambda)$ is doubly stochastic as long as
			\begin{equation}\label{eqn: sumlam}\sum_{i=1}^{n-1}\lambda_i \geqslant -\frac{1}{2}.\end{equation} 
			\item  Suppose that $\lambda_i \geqslant 0$ for all $i \in \{1, \ldots, n-1\}$. Then $P(\Lambda)$ is doubly stochastic as long as
			\begin{equation}\label{eqn: sumlam2}\sum_{i=1}^{n-1}\lambda_i \leqslant \frac{1}{2}.\end{equation} 
		\end{enumerate}
	\end{corollary}
	\begin{proof}
		First we show (1). Assume that (\ref{eqn: sumlam}) holds
		and set
		$S_{j}(k) = \sin\left( \frac{2\pi kj}{n}+\frac{\pi}{4} \right)$
		for any $j \in \{1, \ldots, n-1\}$ and $k\in \{0, \ldots, n-1\}$. Clearly, $S_j(k)S_j(l) \leqslant 1$ for all $k,l \in \{0,\ldots n-1\}$ and since all $\lambda_j \leqslant 0$, we have
		$\lambda_j S_j(k)S_j(l) \geqslant \lambda_j$. Hence, we arrive at the estimate
		\begin{align*}
		\sum_{j=1}^{n-1} \lambda_j  S_j(k)S_j(l) \geqslant&    \sum_{j=1}^{n-1} \lambda_j  \\ \geqslant& -\frac{1}{2}.
		\end{align*}
		To show (2), assume (\ref{eqn: sumlam2}) and use the fact that $S_j(k)S_j(l) \geqslant -1$.
	\end{proof}
	\begin{remark}
		Note that 
		$$0\leqslant \delta_{\min}:=\min\left\{ \sum_{j=0}^{n-1}\lambda_j \colon  \sum_{j=1}^{n-1}\lambda_j S_j(k)S_j(l) \geqslant -\frac{1}{2} \mbox{ for all }k, l \in \{0,\ldots n-1\} \right\} \leqslant \frac{1}{2}.$$
		For the upper bound, let us note that $S_j(k) = 1$ if and only if $n$ is divisible by $8$ and $kj = n \left(\frac{1}{8} + m\right)$
		for some $m \in \mathbb{N} \cup \{0\}$. Thus if $n$ is not a multiple of $8$ then $\delta_{\min} < \frac{1}{2}$.\smallskip 
		
		\noindent For the lower bound we found that if $n=5$ then for  $\sigma = (1, -0.004, -0.002, -0.004, -0.51) $ so that $\sum_{j=0}^{4} \lambda_j = 0.48$ \smallskip 
		we have a negative element in the third row, third column, namely $P(\Lambda)_{22} \approx -0.0005 <0$. These suggest that $\delta_{\min} \in (0.48, 0.5)$.\smallskip
		
		We leave the task of finding the exact value of $\delta_{\min}$ an \emph{open problem}. 
	\end{remark}
	\begin{remark}\label{rem: household}
		One may observe that our matrix $Q$ is orthogonal and symmetric. Therefore, one may try to find different orthogonal symmetric matrix so that Problem \ref{problem} has a~solution for $\delta \in (0, \frac{1}{2})$. This could be achieved by finding a constant $M \in (1, 2)$ that is an~upper bound for the term $n q_j^{(k)}q_j^{(l)}$ in equation \eqref{eqn: mat_gen_entry} for each $j$, $k$, $l \in \{0, 1, \ldots, n-1\}$.\smallskip
		
		A natural example of a matrix that is both orthogonal and symmetric, is a \emph{Householder matrix}, \emph{i.e.},
		\begin{align*}H(v) =& I - 2 \frac{vv^T}{\|v\|^2}\\
		=& 
		\left[\begin{array}{cccccc}
		\beta & \beta & \beta & \ldots & \beta & \beta \\
		\beta & 1-\alpha & -\alpha & \ldots & -\alpha & -\alpha \\
		\beta & - \alpha & 1- \alpha & \ddots & - \alpha & -\alpha\\
		\ldots & \ldots & \ddots & \ddots &\ddots & \ldots \\
		\beta & -\alpha &-\alpha & \ddots & 1-\alpha & - \alpha \\
		\beta & -\alpha &-\alpha & \ldots & -\alpha & 1- \alpha 
		\end{array}\right]
		\end{align*}
		so that $v = (1- \sqrt{n}, 1, \ldots, 1)^T$, $\alpha=\frac{1}{\sqrt{n}(\sqrt{n}-1)}$ and $\beta = \frac{1}{\sqrt{n}}$ (see \cite[Section 4.2]{Zhu}). 
		However, given that $\lambda_j \leqslant 0$ for each $j=1, \ldots, n-1$, the matrix $H(v)\Lambda H(v)$, as $n$ becomes large fails to have non-negative entries as
		\begin{align*}
		(H(v)\Lambda H(v))_{k,k} = \frac{1}{n} + \alpha^2 \sum_{j=1 \ \& \ j\neq k}^{n-1} \lambda_j + (1-\alpha)^2 \lambda_k.
		\end{align*}
		
	\end{remark}
	\iffalse
	\begin{proposition}
		Let $\lambda \in (-1, 1)$ and set $\lambda_j = \lambda^{j}$ TO be finished
	\end{proposition}
	\begin{proof}
		We obtain that
		\begin{align*}
		&\sum_{j=1}^{n-1}\lambda^j \sin\left( \frac{2\pi kj}{n}+\frac{\pi}{4} \right) \sin\left( \frac{2\pi lj}{n}+\frac{\pi}{4} \right)\\
		=&  \frac{1}{2} \left( -\lambda^n + \sum_{j=0}^n \lambda^j \cos\left( \frac{2 \pi j(k-l)}{n} \right)\right) + \frac{1}{2}  \sum_{j=0}^n \lambda^j \sin\left( \frac{2 \pi j(k+l)}{n} \right).
		\end{align*}
		Now 
		\begin{align*}&\sum_{j=0}^n \lambda^j e^{i \frac{2 \pi m j}{n}} \\=& \frac{1}{\lambda^2 - \lambda \cos\left( \frac{2 \pi m}{n} \right)+1} \left( \lambda^{n+1} +1  - \left((\lambda + \lambda^{n+1})\cos\left(\frac{2 \pi m}{n} \right) \right)\right)\\&+ i\frac{1}{\lambda^2 - \lambda \cos\left( \frac{2 \pi m}{n} \right)+1}\left( \left( (\lambda - \lambda^{n+1})\sin\left(\frac{2 \pi m}{n} \right)\right) \right)
		\end{align*}
	\end{proof}
	\fi
	\section{Examples and applications}
	\subsection*{Examples} In this section we provide some examples of Sule\u{\i}manova spectra, $$\sigma_n = (1, \lambda_2, \ldots, \lambda_n),$$ for which $\lambda_2, \ldots, \lambda_n$ add up to $-\frac{1}{2}$ and, thus, yield symmetric doubly stochastic matrices obtained via our construction (\ref{eqn: RWABS}), but do not satisfy any known sufficient conditions (\emph{e.g.},  
	(\ref{eqn: permir}), (\ref{eqn: soules}), \emph{etc.}) to obtain symmetric doubly stochastic matrices.\smallskip
	
	To wit, neither (\ref{eqn: permir}) nor (\ref{eqn: soules}) is satisfied for
	\begin{itemize}
		\item  $\sigma_5 = (1, -0.02, -0.03, -0.05, -0.4)$ (odd dimension);
		\item  $\sigma_6 = (1,-0.01, -0.02, -0.06, -0.08, -0.33)$ (even dimension),
	\end{itemize}
	respectively.
	
	Let $\sigma_n = (1, \lambda_2, \ldots, \lambda_n)$, where $\lambda_2 \geqslant \lambda_3 \ldots \geqslant \lambda_n$.
	The \emph{improved Soules' condition} when $n$ is even, \cite[Theorem 3, Notation 1, Observation 1]{nader}, that is, 
	$$\frac{1}{n} + \frac{1}{n}\lambda_2 + \frac{\frac{n}{2}- \left[\frac{n+2}{4} \right]}{\frac{n}{2}\left[\frac{n+2}{4} \right]}\lambda_4+\sum_{k=1}^{\left[\frac{n+2}{4} \right]-1}\frac{\lambda_{n-4k+4}}{k(k+1)} \geqslant 0$$
	is not satisfied as witnessed by $$\sigma_{10} = (1, -0.01,  -0.01,  -0.025, -0.03,  -0.035, -0.04,  -0.05,  -0.08,  -0.22 )$$
	(the square brackets in the above formula denote the integral part of a real number).\smallskip
	
	Let $n$ be odd, \emph{new condition 1} (\cite[Theorem 4, Notation 1]{nader}; we adapt the naming conventions from the said paper), that is, 
	$$\frac{1}{n} + \frac{n-1}{n(n+1)}\lambda_2 + \frac{\frac{n+1}{2}- \left[\frac{n+3}{4} \right]}{\frac{n+1}{2}\left[\frac{n+3}{4} \right]}\lambda_4+\sum_{k=1}^{\left[\frac{n+3}{4} \right]-1}\frac{\lambda_{n-4k+4}}{k(k+1)} \geqslant 0$$
	is not satisfied as witnessed by the 1/2-Sule\u{\i}manova spectrum $$\sigma_5 =(1,-0.03, -0.03, -0.04, -0.4).$$
	
	Next we give examples of spectra that do not satisfy \emph{New condition 2} (\cite[Theorem 5, Notation 2]{nader})
	which is given by (\ref{eqn: newcond3.0}, \ref{eqn: newcond3.2}, \ref{eqn: newcond3.3}, \ref{eqn: newcond3.1}) depending on the remainder $n \mod 4$. 
	Let  $m$ be an integer greater than 1. If
	\begin{enumerate}
		\item $n=4m$, then
		\begin{equation}\label{eqn: newcond3.0}\frac{1}{n}+\frac{1}{n}\lambda_2+\frac{2}{n}\lambda_4 + \frac{\frac{n}{4} - \left[\frac{n+4}{8}\right]}{\frac{n}{4} \left[\frac{n+4}{8}\right]}\lambda_8 +\sum_{k=1}^{\left[\frac{n+4}{8}\right]-1}\frac{\lambda_{n-8k+8}}{k(k+1)} \geqslant 0
		\end{equation}
		is not satisfied by the 1/2-Sule\u{\i}manova spectrum $$\sigma_{16} = (1, -0.003,  -0.003,  -0.004,  -0.007,  -0.009,  -0.02,   -0.0209,-0.021,$$ $$  -0.024,
		-0.026,  -0.035,  -0.042,  -0.076,  -0.0811, -0.128);$$
		\item $n=4m+2$, then
		\begin{equation}\label{eqn: newcond3.2}\frac{1}{n}+\frac{1}{n}\lambda_2+\frac{2(n-2)}{n(n+2)}\lambda_4 + \frac{\frac{n+2}{4} - \left[\frac{n+6}{8}\right]}{\frac{n+2}{4} \left[\frac{n+6}{8}\right]}\lambda_8 +\sum_{k=1}^{\left[\frac{n+6}{8}\right]-1}\frac{\lambda_{n-8k+8}}{k(k+1)} \geqslant 0 \end{equation}
		is not satisfied by the 1/2-Sule\u{\i}manova spectrum $$\sigma_{10}=(1, -0.01, -0.01, -0.01, -0.02, -0.02, -0.04, -0.07, -0.1  ,-0.22);$$
		\item $n=4m+3$, then
		\begin{equation}\label{eqn: newcond3.3}\frac{1}{n}+\frac{n-1}{n(n+1)}\lambda_2+\frac{2}{n+1}\lambda_4 + \frac{\frac{n+1}{4} - \left[\frac{n+5}{8}\right]}{\frac{n+1}{4} \left[\frac{n+5}{8}\right]}\lambda_8 +\sum_{k=1}^{\left[\frac{n+5}{8}\right]-1}\frac{\lambda_{n-8k+8}}{k(k+1)} \geqslant 0 \end{equation}
		is not satisfied by the 1/2-Sule\u{\i}manova spectrum $$\sigma_{11} = (1, -0.001, -0.004, -0.01,  -0.01,  -0.012, -0.013, -0.05,  -0.09,  -0.11,  -0.2 );$$
		\item $n=4m+1$, then
		\begin{equation}\label{eqn: newcond3.1}\frac{1}{n}+\frac{n-1}{n(n+1)}\lambda_2+\frac{2(n-1)}{(n+1)(n+3)}\lambda_4 + \frac{\frac{n+3}{4} - \left[\frac{n+7}{8}\right]}{\frac{n+3}{4} \left[\frac{n+7}{8}\right]}\lambda_8 +\sum_{k=1}^{\left[\frac{n+7}{8}\right]-1}\frac{\lambda_{n-8k+8}}{k(k+1)} \geqslant 0 \end{equation}
		is not satisfied by the 1/2-Sule\u{\i}manova spectrum $$\sigma_9 = (1, -0.006, -0.018, -0.02,  -0.028, -0.028, -0.053, -0.105, -0.242).$$
	\end{enumerate}
	\emph{New condition 3} (\cite[Conjecture 1, Example 1]{nader}) that for $n=26$ takes the form
	$$\frac{1}{26}+\frac{1}{26}\lambda_2 + \frac{6}{13\cdot 7}\lambda_4 + \frac{3}{28}\lambda_8+\frac{1}{4}\lambda_{16}+\frac{1}{2}\lambda_{26}\geqslant 0$$
	is not satisfied by the 1/2-Sule\u{\i}manova spectrum $$\sigma_{26} = (1, \underbrace{-0.004}_{\lambda_2}, -0.005, \underbrace{-0.006}_{\lambda_4}, -0.007, -0.01,  -0.01,  \underbrace{-0.011}_{\lambda_8}, -0.011, -0.011, -0.012,$$
	$$-0.012, -0.015, -0.015, -0.016, \underbrace{-0.017}_{\lambda_{16}}, -0.019, -0.02,  -0.022, -0.022, -0.025,
	-0.028,$$ $$-0.028, -0.032, -0.069, \underbrace{-0.073}_{\lambda_{26}}).$$
	
	\subsection*{Applications to random generation of doubly stochastic matrices} Let $n \in \mathbb{N}$ and $\alpha \in \left[-\frac{1}{2}, \frac{1}{2}\right]$. Our construction provides a simple way to randomly generate symmetric doubly stochastic matrices via their spectra. Namely, 
	let $X_1, \ldots, X_{n-1}$ be independent random variables having probability distributions supported on $[0, 1]$ and let us consider $S_n := X_1+ \ldots + X_{n-1}$. Set
	$$ \lambda_i = \alpha \frac{X_i}{S_n}\quad (i \in \{1, \ldots, n-1\}). $$ 
	Then $\sigma = (1, \lambda_1, \ldots, \lambda_{n-1})$
	is a spectrum of a symmetric doubly stochastic matrix and the corresponding matrix may be obtained via (\ref{eqn: RWABS}).
	More algorithms to generate doubly stochastic matrices (not necessarily symmetric) can be found in \cite{capel}.
	For an elaborate discussion on spectral properties of random doubly stochastic matrices we refer the reader to \cite{ng}.
	%Consider $$p_{{11}} = \frac{1}{n}\left( 1 -  \frac{1}{S_n}\sum_{j=1}^{n-1}X_j \sin^2\left( \frac{2 \pi j}{n} + \frac{\pi}{4}\right) \right)$$ 
	
	\subsection*{Supplementary material} We supplemnt the material with a Python code organised in a  \texttt{JupyterNotebook} available at \begin{center}\url{https://github.com/Nty24/DoublyStochasticMatricesGenerator}\end{center} that generates further examples and counterexamples in the spirit of Section 3. 
	
	\subsection*{Acknowledgements} The first-named author would like to thank his colleagues, James Burridge, for the idea to investigate the eigenvectors of symmetric random walks, and Ittay Weiss, for great discussions that led to Remark \ref{rem: household}. The authors would like to thank Pietro Paparella for  clarifying the history of the results on realising normalised  Sule\u{\i}manova spactra. M.G.'s visit to Prague in August 2019, during which some part of the project was completed, was supported by funding received from GA\v{C}R project 19-07129Y; RVO 67985840, which is acknowledged with thanks. The work was supported by the Iuventus Plus grant IP2015 035174.

\end{document}